\documentclass[11pt]{amsart}

\newtheorem{theorem}{Theorem}[section]

\newtheorem{corollary}[theorem]{Corollary}

\newtheorem{definition}[theorem]{Definition}
\newtheorem{lemma}[theorem]{Lemma}

\newtheorem{proposition}[theorem]{Proposition}

\theoremstyle{definition}
\newtheorem{remark}[theorem]{Remark}

\usepackage{bbm}

\usepackage{hyperref,cleveref,graphics,mathrsfs}
\usepackage{color,graphicx,shortvrb}

\usepackage{tikz}

\usepackage{enumerate}
\usepackage{amsmath}
\usepackage{amssymb}

\usepackage{blindtext}
\usepackage{scrextend}
\addtokomafont{labelinglabel}{\sffamily}

\newcommand{\tri}[1]{| \! | \! | {#1} | \! | \! |}

\newcommand{\spn}{\mathrm{span}}
\newcommand{\vr}{\varepsilon}

\newcommand{\ball}{\mathbf{B}}

\newcommand{\ff}{{\mathfrak{F}}}
\newcommand{\co}{{\mathbf{c}}}
\newcommand{\dd}{{\mathbf{d}}}
\newcommand{\gm}{{\boldsymbol{\gamma}}}
\newcommand{\xx}{{\mathbf{X}}}
\newcommand{\yy}{{\mathbf{Y}}}

\newcommand{\N}{\mathbb{N}}

\newcommand{\R}{\mathbb{R}}

\newcommand{\A}{\mathcal{A}}

\newcommand{\triple}[1]{{\left\vert\kern-0.25ex\left\vert\kern-0.25ex\left\vert #1 
    \right\vert\kern-0.25ex\right\vert\kern-0.25ex\right\vert}}

\begin{document}

\numberwithin{equation}{section}

\title{Renorming AM-spaces}

\author{T.~Oikhberg and M.A.~Tursi}

\address{ Dept.~of Mathematics, University of Illinois, Urbana IL 61801, USA}
\email{oikhberg@illinois.edu, gramcko2@illinois.edu}



\parindent=0pt
\parskip=3pt

\begin{abstract}
We prove that any separable AM-space $X$ has an equivalent lattice norm for which no non-trivial surjective lattice isometries exist.
Moreover, if $X$ has no more than one atom, then this new norm may be an AM-norm.
As our main tool, we introduce and investigate the class of so called Benyamini spaces, which ``approximate'' general AM-spaces.
\end{abstract}

\maketitle
\thispagestyle{empty}

\section{Introduction}\label{s:intro}

The question of renormings has been extensively studied in the Banach space literature.
The goal is to equip a prescribed Banach space with an equivalent norm in a way that alters its isometric properties in a certain desirable way (the isomorphic properties meanwhile remain the same).
Many results of this type appear in \cite{Diestel75}; for more modern treatment see \cite{FA_inf_dim_geometry} or \cite{BS theory}.

We are interested in producing a renorming with a prescribed group of isometries (throughout this paper, all isometries are assumed to be linear and surjective unless specified otherwise).
One of the first results appeared in \cite{Be}; there, it was shown that any separable real Banach space can equipped with an equivalent norm for which there are only two isometries -- the identity and its opposite.
The separability assumption was later removed in \cite{Ja}.
More recent papers \cite{FeG_TAMS10}, \cite{FeR_EM11}, \cite{FeR_DU13} deal with renorming a separable Banach space in a way that produces a prescribed group of isometries.

In this work, we consider lattice renormings of separable AM-spaces.
Recall that an AM-space is a Banach lattice in which the equality $\| x \vee y \| = \max\{ \|x\|, \|y\| \}$ for any positive $x$ and $y$; a lattice norm with this property is called an AM-norm.
We also restrict oursleves to lattice isometries -- that is, surjective (linear) isometries which preserve lattice operations. Our main result is:

\begin{theorem}\label{t:renorm_AM}
Suppose $(X, \| \cdot \|)$ is a separable AM-space, and $c > 1$. Then $X$ can be equipped with an equivalent lattice norm $\triple{ \cdot }$ so that $\| \cdot \| \leq \triple{ \cdot } \leq c \| \cdot \|$, and the identity map is the only lattice isometry on $(X. \triple{ \cdot })$.
If $X$ has no more than one atom, then $\tri{ \cdot }$ can be chosen to be an AM-norm.
\end{theorem}

The restriction on the number of atoms is essential; see Remark \ref{r:assumptions_needed}.

The proof of Theorem \ref{t:renorm_AM} proceeds in two steps. In Section \ref{s:benyamini}, we
introduce a new class of AM-spaces, which we call ``Benyamini spaces,'' for their original discoverer \cite{Ben}.
We establish that any separable AM-space can be transformed into a Benyamini space with arbitrarily small distortion; this result may be interesting in its own right.
In Section \ref{s:am_renormings_benyamini} we renorm a Benyamini space, eliminating all non-trivial isometries (the new norm may cease to be an AM-norm if more than one atom is present).
The proof of Theorem \ref{t:renorm_AM} then follows by combining Proposition \ref{p:benyamini-approximation} and Theorem \ref{t:new_theorem_about_Benyamini}.

Throughout this paper, we use the standard functional analysis facts and notation. For more detail, the reader is referred to e.g.~\cite{LT2} and \cite{M-N}. All spaces are assumed to be separable, and the field of scalars is that of real numbers. For a normed space $Y$, we use the notation $\ball(Y) = \{ y \in Y : \|y\| \leq 1\}$.
If $Y$ is an ordered space and $A \subset Y$, we denote by $A_+$ the positive part of $A$ -- that is, $\{a \in A : a \geq 0\}$.

\section{Benyamini spaces}\label{s:benyamini}

Here we investigate a class of AM-spaces -- the Benyamini spaces. Such spaces are flexible: any separable AM-space can be transformed into a space of this form (Proposition \ref{p:benyamini-approximation}).
On the other hand, Benyamini spaces can be easily analyzed, since they have a concrete representation, similar to a $C(K)$ space.
In particular, we describe atoms in, and duals of, such spaces in Subsections \ref{ss:atoms_benyamini} and \ref{ss:dual}, respectively.

\subsection{Definition and basic properties}

\begin{definition}\label{d:benyamini}
We say that a Banach lattice $X$ is a \emph{$C$-Benyamini space} (the constant $C > 1$ will often be omitted) if it is a sublattice of $C(K)$, where:
\begin{enumerate}
\item $K$ is the one point compactification of the union of mutually disjoint compact sets $K_n$ ($K = (\cup_n K_n) \cup \{\infty\}$).
\item $X \subset C_0(K)$ -- that is, any $x \in X$ vanishes at $\infty$.
\item $X$ separates points for each $K_n$ (that is, for all $t,s\in K_n$, there exists $x \in X$ such that $x(t) \neq x(s)$).
\item If $t \in K_m, \ s \in K_n$, 
and for all $x \in X$,  $x(t) = \lambda x(s)$ for some fixed $\lambda$, then $\lambda = C^{n-m}$.
\end{enumerate}
\end{definition}

Note that, if $X$ is separable, then each $K_n$ is metrizable (due to (3)). Consequently, $K$ is metrizable.

We begin by establishing some properties of Benyamini spaces.

\begin{lemma}\label{l:homeo-subsets}
	Let $K$, $C$, and $X \subseteq C_0(K)$ be as above.  For $n\neq m$, define $D(m,n)$ by \[ D(m,n) := \big\{t \in K_m: \exists s \in K_n \text{ such that } \forall x\in X, x(t) = C^{n-m} x(s)  \big\}. \]
	Then $D(n,m)$ is closed and homeomorphic to $D(m,n)$.  
\end{lemma}

\begin{proof}
Define $\phi_{mn} : D(m,n) \to D(n,m)$ by setting $\phi_{mn}(t)$ to be the unique $s \in D(n,m)$ with the property that $x(t) = C^{n-m} x(s)$ for any $x \in X$.  
	Because $X$ separates points for each $K_n,$ $\phi_{mn}$ is well defined and injective.  By definition of $D(m,n)$, it is bijective, and $ \phi_{nm} = \phi_{mn}^{-1}$. To show continuity, suppose $t_k \rightarrow t \in D(m,n)$ with $t_k \in D(m,n)$.  Suppose there exists a subsequence $s_j = \phi_{mn}(t_{k_j}) \rightarrow s \neq \phi_{mn}(t) $ (we can limit ourselves to such a case since $K_m$ is compact).  Then for all $x\in X$,
	\[ x(s) = \lim x(s_j) = \lim\limits_{j}  C^{m-n}x(t_{k_j}) = C^{m-n} x(t) = x (\phi_{m,n}(t)) , \] 
	which is a contradiction, since $X$ separates points. 
	To prove that $D(m,n)$ is closed, suppose $t_k \rightarrow t$.  Then for all $x\in X,$ \[  x(t) = \lim\limits_{k} x(t_k) = C^{n-m} \lim\limits_{k} x(\phi_{mn}(t_k)).  \]
	By compactness, we assume that $\phi_{mn}(t_k) \rightarrow s\in K_n$.  Hence for all $x\in X, x(t) = C^{n-m}x(s)$, so $s\in D(m,n)$.
	\end{proof}

 
The importance of Benyamini spaces stems from the fact that any separable AM-space can be ``approximated'' by a Benyamini space.

\begin{proposition}\label{p:benyamini-approximation}
If $X$ is a separable AM-space, then for every $C > 1$ there exists a Benyamini space $X'$ and a surjective lattice isomorphism $\Phi : X \to X'$ so that for all $x\in X$, $\|x \| \leq \| \Phi(X) \| \leq C \| x \|$.
\end{proposition} 

The proof below is similar to that of \cite[Lemma 1]{Ben}.

\begin{proof}
	We can assume that $X \subset C(H)$ for some Hausdorff compact $H$.
	First, as in \cite{Ben}, we consider the set $F:= \cap_{x \in X} x^{-1}(0)$.  If $F \neq \emptyset$,  identify $F$ with a single point $z$ by passing from $K$ to $K/F$.  Let $x_n$ be a dense sequence in $\ball(X)_+$. Let
	$\psi = (C-1)\sum_{n=1}^\infty C^{-n} x_n$. Clearly $\psi$ belongs to $X$.
	
	Let $H_n = \{ t\in H:  C^{-n} \leq \psi(t) \leq C^{-n+1}  \}$.  If infintely many $H_n$'s are non-empty, let $\widetilde{H}_n$ be disjoint copies of $H_n$, and let $\widetilde{H} = (\cup_n \widetilde{H}_n) \bigcup \{\infty\}$ be the one point compactification of $\bigcup \widetilde{H}_n$. Otherwise, let $\widetilde{H} = \bigcup_n \widetilde{H_n}$.	Define the map $\Psi:\widetilde{H} \rightarrow H$ sending $\widetilde{H}_n$ to $H_n$ and $\infty$ to $z$. Note that if $F$ is empty, then $\psi(t) > 0$ for all $t\in K$, and since $\psi$ itself is continuous, its image is compact and so must be bounded below; then $H_n = \emptyset$ for $n$ large enough. Otherwise, $\psi$ vanishes only at $z$. In either case, $\Psi$ is a continuous surjection from $\widetilde{H}$ onto $H$, which implies that $C(H)$ embeds into  $C(\widetilde{H})$ isometrically via the map $x \mapsto \widetilde{\Psi}x := x \circ \Psi$. 
	
	Now define a lattice isomorphism $U : C_0(\widetilde{H}) \to C_0(\widetilde{H})$ by setting, for $x \in C_0(\widetilde{H})$, $[Ux](\infty) = 0$, and $\displaystyle [Ux](t) = \frac{C^{1-n} x(t)}{(\widetilde{\Psi}\psi)(t)}$. Observe that $\|Ux\| \leq \|x\| \leq C \|x\|$.
	Then $T = U \circ \widetilde{\Psi}$ is a lattice homomorphism, and $Y = T(X)$ is a sublattice of $C_0(\widetilde{H})$. We claim that, if $t \in \widetilde{H}_m$ and $s \in \widetilde{H}_n$ are such that $y(t) = \lambda y(s)$ for any $y \in Y'$, then $\lambda = C^{n-m}$. Indeed, $y = Tx$ for some $x \in X$, so
	$$
	\lambda = \frac{y(t)}{y(s)} = \frac{C^{1-m}x(\Psi(t))}{\psi(\Psi(t))} \cdot \frac{\psi(\Psi(s))}{C^{1-n} x(\Psi(s))} =
	C^{n-m} \cdot \frac{x(\Psi(t))}{x(\Psi(s))} \cdot \frac{\psi(\Psi(s))}{\psi(\Psi(t))} .
	$$
	From this, it follows that $x(\Psi(t))/x(\Psi(s))$ is a constant on $X$. Either $\Psi(t)= \Psi(s)$, or $t' = \Psi(t)$ and $s' = \Psi(s)$ are ``defining points'' for $X \subset C(H)$ -- that is, $x(t')/x(s')$ is independent of $x \in X$. Either way, $\lambda = C^{n-m}$.
	
	Finally, we transform the sets $\widetilde{H}_n$ into sets $K_n$, whose points are separated by $X'$.
	By the preceding paragraph, if $t, s \in \widetilde{H}_n$ are such that $y(t) = \lambda y(s)$ for any $y \in Y$, then $\lambda = 1$.
	Define an equivalence relation on $\widetilde{H}$: $t \sim s$ if for all $y\in Y, \ y(s)  = y(t)$. Clearly the equivalence classes are closed, hence each quotient space $K_n := \widetilde{H}_n/\sim$ is compact.
	Identify $\widetilde{H}/\sim$ with $K = (\cup_n K_n) \cup \{\infty\}$, which is the one-point compactification of $\cup_n K_n$. Define $\Phi : Y \to C_0(K)$ by setting, for $y \in Y$, $[\Phi y]([t]) = y(t)$, where $[t]$ is the equivalence class of $t$.
	Clearly $\Phi$ is a lattice isometry. $X' = \Phi(Y)$ is a Benyamini space, and $\Phi \circ T : X \to X'$ is a lattice isomorphism with desired properties.
	\end{proof}

	


	\begin{remark}\label{r:benyamini preserve isometries}
	 The Benyamini space $X'$, constructed from $X$ using Proposition \ref{p:benyamini-approximation}, may have a different group of isometries.
	 We do not know whether the Benyamini space can be constructed while preserving the group of isometries (or even a subgroup thereof).
	\end{remark}

\subsection{Extension of functions in Benyamini spaces}\label{ss:extensions}

We say that a function $x \in C(K_M \cup \ldots \cup K_N)$ is \emph{consistent} if $x(s) = C^{n-m} x(\phi_{mn}(s))$ whenever $s \in D(m,n)$, with $M \leq n,m \leq N$.
We shall say that a family of functions $x_n \in C(K_n)$ ($M \leq n \leq N$) is \emph{consistent} if the function $x \in C(K_M \cup \ldots \cup K_N)$, defined via $x|_{K_n} = x_n$, is consistent.

\begin{proposition}\label{p:consistent_extension}
(1) If $L \leq N$, and $x \in C(K_L \cup \ldots \cup K_N)$ is a consistent function, then there exists $\tilde{x} \in X$ so that $\tilde{x}|_{K_1 \cup \ldots \cup K_N} = x$, and, for $j \notin \{L, \ldots, N\}$, $\sup_{K_j} |\tilde{x}| \leq \max_{L \leq i \leq N} C^{i-j} \sup_{K_i} |x|$.

(2) If, furthermore, $y \in X_+$ is such that $0 \leq x \leq y$ on $K_L \cup \ldots \cup K_N$, then $\tilde{x}$ can be selected in such a way that, in addition, $0 \leq \tilde{x} \leq y$.
\end{proposition}

\begin{remark}
 In a similar fashion, one can show that if $y, z \in X$ are such that $z \leq x \leq y$ on $K_M \cup \ldots \cup K_N$, then $\tilde{x}$ can also be selected in such a way that $z \leq \tilde{x} \leq y$.
\end{remark}

The proof of Proposition \ref{p:consistent_extension} is obtained by combining Lemmas \ref{l:consistent_extension_down} and \ref{l:consistent_extension_up}.

First we deal with ``downward'' extensions.

\begin{lemma}\label{l:consistent_extension_down}
(1) If $x \in C(K_1 \cup \ldots \cup K_N)$ is a consistent function, then there exists $\tilde{x} \in X$ so that $\tilde{x}|_{K_1 \cup \ldots \cup K_N} = x$, and, for $j > N$, $\sup_{K_j} |\tilde{x}| \leq \max_{1 \leq i \leq N} C^{i-j} \sup_{K_i} |x|$.

(2) If, furthermore, $y \in X_+$ is such that $0 \leq x \leq y$ on $K_1 \cup \ldots \cup K_N$, then $\tilde{x}$ can be selected in such a way that, in addition, $0 \leq \tilde{x} \leq y$.
\end{lemma}

\begin{proof}
(1) We define $\tilde{x}$ recursively. Suppose $\displaystyle \tilde{x}|_{K_1 \cup \ldots \cup K_{M-1}}$, with $M-1 \geq N$, has already been defined in such a way that $\sup_{K_j} |\tilde{x}| \leq \max_{1 \leq i \leq N} C^{i-j} \sup_{K_i} |x|$ whenever $N < j < M$.
Define now $\tilde{x}$ on $K_M$. If $t \in D(M,j)$ for some $j < M$, set $x(t) = C^{j-M} x(\phi_{Mj}(t))$. Note that $x$ is well-defined on $\cup_{j < M} D(M,j)$: if $t \in D(M,j) \cap D(M,i)$, then $C^{j-M} x(\phi_{Mj}(t)) = C^{i-M} x(\phi_{Mi}(t))$.
Also, for such $t$, $|x(t)| \leq \max_{1 \leq i \leq N} C^{i-M} \sup_{K_i} |x|$.

Moreover, $\tilde{x}$ is continuous on the closed set $D(M,j)$ for every $j<M$, and thus also on $\cup_{j < M} D(M,j)$. Extend $\tilde{x}$ to a continuous function on $K_M$ without increasing the $\sup$-norm.


Finally, set $\tilde{x}(\infty) = 0$. The function $\tilde{x}$ thusly defined belongs to $X$. Indeed, it is continuous on each of the sets $K_n$, and also at $\infty$, given that $\sup_{K_j} |\tilde{x}| \leq {\mathrm{const}} \cdot C^{-j}$. Finally, if $t \in D(n,m)$, then $\tilde{x}(t) = C^{m-n} \tilde{x}(\phi_{nm}(t))$. \\

(2) 
Modify the recursive process from part (1).
Suppose $\tilde{x}|_{K_1 \cup \ldots \cup K_{M-1}}$, where $M-1 \geq N$, has already been defined in such a way that $0 \leq \tilde{x} \leq y|_{K_1 \cup \ldots \cup K_{M-1}}$
on ${K_1 \cup \ldots \cup K_{M-1}}$ and
$\sup_{K_j} \tilde{x} \leq \max_{1 \leq i \leq N} C^{i-j} \sup_{K_i} x$ whenever $N < j < M$.
Define now $\tilde{x}$ on $K_M$. If $t \in D(M,j)$ for some $j < M$, set $x(t) = C^{j-M} x(\phi_{Mj}(t))$.
As before, observe that $x$ is well-defined on $\cup_{j < M} D(M,j)$.
Clearly, for $t \in D(M,j)$,
$$
0 \leq \tilde{x}(t) \leq y(t) ,
{\textrm{  and   }}
\tilde{x}(t) \leq \max_{1 \leq i \leq N} C^{i-M} \sup_{K_i} x .
$$
Also, $\tilde{x}|_{\cup_{j < M} D(M,j)}$ is continuous.
Therefore, we can find $u \in C(K_M)$
so that
$$
\sup_{K_M} |u| = \sup_{\cup_{j < M} D(M,j)} |\tilde{x}| \leq \max_{1 \leq i \leq N} C^{i-M} \sup_{K_i} |x|.
$$
To define $\tilde{x}$ on $K_M$, set $\tilde{x} = u \wedge y$.
\end{proof}

We shall use the notation $K_n' = K_n \backslash (\cup_{m < n} D(n,m))$, and  $K' = \cup_n K_n'$ (note that these sets are open).

In a manner similar to the preceding lemma, one can prove:

\begin{lemma}\label{l:extend_from_two roots}
Suppose $m \leq n$, $t \in K_m'$, $s \in K_n'$, and $U \subset K_m'$, $V \in K_n'$ are disjoint open sets with the property that $t \in U \subset \overline{U} \subset K_m'$ and $s \in V \subset \overline{V} \subset K_n'$. Then for $\alpha, \beta \in [0,\infty)$, there exists $x \in X_+$ so that:
\begin{enumerate}
 \item 
 For $j < m$, $x|_{K_j} = 0$.
 \item
 $x(t) = \alpha$, $x(s) = \beta$, $x \leq \alpha$ on $U$, and $x \leq \beta$ on $V$.
 \item
 If $m < n$, then $x|_{K_m \backslash U} = 0$.
 \item
 If $m < n$, then for $m < j < n$, $0 \leq x|_{K_j} \leq C^{m-j} \alpha$.
 \item
 On $K_n$, $0 \leq x \leq C^{m-n} \alpha \vee \beta$. 
 \item
 For $j > n$, $0 \leq x|_{K_j} \leq (C^{m-j} \alpha )\vee( C^{n-j} \beta$).
\end{enumerate}
\end{lemma}

\begin{proof}
We shall consider the case of $m < n$ (that of $m = n$ is handled similarly).
In light of Lemma \ref{l:consistent_extension_down}, it suffices to construct a consistent family of functions $x_j \in C(K_j)$, with $j \leq n$, satisfying the properties listed above. For $j < m$, simply set $x_j = 0$.
Define $x_m \in C(K_m)_+$ which vanishes outside of $U$ and satisfies $0 \leq x \leq \alpha = x(t)$.

Use Lemma \ref{l:consistent_extension_down} to find  $x_j \in C(K_j)$ so that the family $(x_j)_{j < n}$ is consistent and $x_j \leq C^{m-j} \alpha$.

Define $x_n \in C(K_n)$ in such a way that:
\begin{enumerate}
 \item $x_n = 0$ on $\partial V$, and $0 \leq x_n \leq \beta = x_n(s)$  on $V$.
 \item $x_n(t) = C^{j-n} x_j(\phi_{nj}(t))$ whenever $t \in D(n,j)$ for some $j < n$.
\end{enumerate}
Such a function $x_n$ exists, since $\overline{V}$ is disjoint from $\cup_{j < n} D(n,j)$. Furthermore, the family $(x_j)_{j \leq n}$ is consistent. To define $x_j$ for $j > n$, again invoke Lemma \ref{l:consistent_extension_down}.
\end{proof}

Next we consider ``upward'' extensions.

\begin{lemma}\label{l:consistent_extension_up}
(1) If $L \leq N$, and $x \in C(K_L \cup \ldots \cup K_N)$ is a consistent function, then there exists a consistent $\tilde{x} \in C(K_1 \cup \ldots \cup K_N)$ so that $\tilde{x}|_{K_L \cup \ldots \cup K_N} = x$, and for $j < L$, $\sup_{K_j} |\tilde{x}| \leq \max_{L \leq i \leq N} C^{i-j} \sup_{K_i} |x|$.

(2) If, furthermore, $y \in X_+$ is such that $0 \leq x \leq y$ on $K_L \cup \ldots \cup K_N$, then $\tilde{x}$ can be selected in such a way that, in addition, $0 \leq \tilde{x} \leq y$.
\end{lemma}

\begin{proof}
We only prove (1), as (2) is handled similarly (compare with the proof of Lemma \ref{l:consistent_extension_down}).

Define $\tilde{x}$ recursively. Suppose $\tilde{x}|_{K_{M+1} \cup \ldots \cup K_N}$ ($M+1 \leq L$) has already been defined in such a way that $\sup_{K_j} |\tilde{x}| \leq \max_{L \leq i \leq N} C^{i-j} \sup_{K_i} |x|$ whenever $M < j < N$.
Now define $\tilde{x}$ on $K_M$. If $t \in D(M,j)$ for some $j \in \{M+1, \ldots, N\}$, set $x(t) = C^{j-M} x(\phi_{Mj}(t))$. Note that $x$ is well-defined on $\cup_{N \leq j < M} D(M,j)$: if $t \in D(M,j) \cap D(M,i)$, then $C^{j-M} x(\phi_{Mj}(t)) = C^{i-M} x(\phi_{Mi}(t))$.
Also, for such $t$, $|x(t)| \leq \max_{1 \leq i \leq N} C^{i-M} \sup_{K_i} |x|$. 

As $\tilde{x}|_{\cup_{M < j \leq N} D(M,j)}$ defined above is continuous, we can extend it to the whole $K_M$, without increasing the $\sup$-norm.
\end{proof}

\subsection{Atoms in a Benyamini space}\label{ss:atoms_benyamini}

\begin{definition}\label{d:hereditarily_isolated}
A point $k \in K'$ is called \emph{hereditarily isolated} if it is an isolated point of $K_n'$ for some $n \in \N$, and $\phi_{nm}(k)$ is isolated in $K_m$ whenever $k \in D(n,m)$.
\end{definition}

For a point $k$ like this, we can define a function $\theta_k \in X$ by setting $\theta_k(k) = 1$, $\theta_k(\phi_{nm}(k)) = C^{n-m}$ whenever $k \in D(n,m)$, and $\theta_k(t) = 0$ otherwise. Clearly $\theta_k$ is a normalized atom in $X$. Our next result claims that all atoms in $X$ are of this form.

\begin{proposition}\label{p:atomes_in_benyamini}
 If $x \in X$ is a normalized atom, then $x = \theta_k$ for some hereditarily isolated point $k$.
\end{proposition}

\begin{proof}
 Suppose $x \in X$ is a normalized atom. Find $k \in K_n'$ such that $x(k) = 1$. We now prove that $k$ is a hereditarily isolated point and that $x = \theta_k$. In particular, we must show that if $k \in D(n,m)$, then $\phi_{nm}(x)$ is isolated in $K_m$ (note that here, $m \geq n$ necessarily).
 
 Suppose, for the sake of contradiction, that $k_m = \phi_{nm}(k)$ is not isolated in $K_m$ for some $m$. Find the smallest such $m$.
 Find distinct $a_1, a_2 \in K_m$ so that $x(a_1), x(a_2) > 1/2$. 
 Find $y \in C(K_m)$ so that $0 \leq y \leq x|_{K_m}$, $y_1(a_1) = \frac12$, and $y(a_2) = 0$.
 By Proposition \ref{p:consistent_extension}, there exists $\tilde{y} \in [0,x] \subset X$ such that $\tilde{y}|_{K_m} = y$. By our choice of $y$, $\tilde{y}$ cannot be a scalar multiple of $x$.
 Thus $x$ is not an atom, which is the desired contradiction.
\end{proof}

	\subsection{The dual of a Benyamini space}\label{ss:dual}

\begin{lemma}\label{l:am-dual}
	Let $X$ and $K'$ be as above.  Then $X^*$ is lattice isometric to $M(K')$.
\end{lemma}

\begin{proof}
Any measure on $K'$ determines a linear functional on $X$; this gives rise to a contraction ${\mathbf{i}} : M(K') \to X^*$. We prove that ${\mathbf{i}}$ is a surjective isometry by showing that any $x^* \in X^*$
can be represented by $\mu \in M(K')$ with $\|\mu\| \leq \|x^*\|$.
By the Hahn-Banach Theorem, $x^*$ extends to a functional on $C(K)$ of the same norm; the latter is implemented by a measure $\mu \in M(K)$, with $\|\mu\| = \|x^*\|$. By removing a point mass at $\infty$, we can and do assume that $\mu$ lives on $\cup_n K_n$.

We claim that $\mu$ vanishes on $K \backslash K'$. Indeed, otherwise find the smallest value of $n$ for which $\mu$ does not vanish on $K_n \backslash K_n'$; then $\mu|_{\cup_{j < n} D(n,j)} \neq 0$. Find the smallest $j$ so that $\mu|_{D(n,j)} \neq 0$. Then the measure
$$
\mu' = \mu - \mu \big|_{D(n,j)} + C^{j-n} \mu \big|_{D(n,j)} \circ \phi_{jn} 
$$
implements the same functional $x^*$; here, for $x \in C(K)$, we define $\big[\mu \big|_{D(n,j)} \circ \phi_{jn}\big](x)$ to be $\mu \big|_{D(n,j)} \big(x\big|_{D(j,n)} \circ \phi_{nj}\big)$.
Note that $\mu'(E) = \mu(E) + C^{j-n} \mu(\phi_{jn} (E))$ for $E \subset D(j,n)$, $\mu'(E) = 0$ for $E \subset D(n,j)$, and $\mu'(E) = \mu(E)$ if $E$ is disjoint from $D(n,j) \cup D(j,n)$.
Furthermore,  $\mu'|_{K_m} = \mu|_{K_m}$ for $m \notin \{j, n\}$, $\mu'|_{K_n} = \mu|_{K_n \backslash D(n,j)}$, and $\mu'|_{K_j} = \mu|_{K_j} + C^{j-n} \mu \big|_{D(n,j)} \circ \phi_{jn}$.
It follows that
$$
\| \mu'|_{K_n} \| = \| \mu|_{K_n} \| - \| \mu|_{D(n,j)} \| ,
$$
while
$$
\| \mu'|_{K_j} \| \leq \| \mu|_{K_j} \| + C^{j-n} \| \mu|_{D(n,j)} \| ,
$$
Therefore,
\begin{align*}	
\|\mu'\|&  = \sum_i \|\mu'|_{K_i}\| =
\| \mu'|_{K_n} \| + \| \mu'|_{K_j} \|  + \sum_{i \notin \{j,n\}} \|\mu'|_{K_i}\| \\
& \leq (C^{j-n} - 1 )\| \mu|_{D(n,j)} \| + \sum_i \|\mu|_{K_i}\| <
\sum_i \|\mu|_{K_i}\| = \|x^*\| ,
\end{align*}
a contradiction.

It is clear that the map ${\mathbf{i}}$ is positive (a positive measure generates a positive functional). We now show that ${\mathbf{i}}$ is bipositive: if $\mu \in M(K')$ is not a positive measure, then the corresponding functional is not positive either. We can write $\mu = (\mu_n)$, with $(\mu_n)$ concentrated on $K_n'$. Note that $\|\mu\| = \sum_n \|\mu_n\|$.

Find $N \in \N$ so that $\mu_n \geq 0$ for $n < N$, but $\mu_N$ is not positive. By the regularity of the measure $\mu_N$, we can find a positive $x_N \in C(K_N)$, vanishing on $\cup_{j<N} D(N,j)$, so that $\mu_N(x_N) < 0$. By scaling, we can and do assume that $\|x_N  \|_\infty = 1$.
Let $\delta = - \mu_N(x_N)/3$. Find $M > N$ so that $\sum_{j > M} C^{N-j} \|\mu_j\| < \delta$.

For $j < N$, let $x_j$ be the zero function on $K_j$. For $N< j \leq M$, find an open set $U_j \subset K_j$ containing $\cup_{i < j} D(j,i)$ with $\|\mu_j|_{U_j}\| < \delta/M$. 
Now use Lemma \ref{l:consistent_extension_down} to define, recursively, a consistent family of functions $x_j$
($j > N$) so that $\|x_j\| \leq C^{N-j}$ and $x_j$ vanishes outside of $U_j$ for $N < j \leq M$.
By our choice of $U_j$, we have $|\mu_j(x_j)| \leq \delta C^{N-j} /M$ for $N < j \leq M$; for $j > M$, we have $|\mu_j(x_j)| \leq \delta C^{N-j} \|\mu_j\|$. 
Merge all the $x_j$'s into a function $x \in X$. Then
\begin{align*}
\mu(x)
&
\leq
\mu_N(x_N) + \sum_{j > N} | \mu_j(x_j)| \leq -3\delta + \sum_{j=N+1}^M C^{N-j} \frac{\delta}{M} + \sum_{j > M} C^{N-j} \|\mu_j\|
\\
&
<
-3\delta + (M-N+1) \frac{\delta}{M} + \sum_{j > M} C^{N-j} \|\mu_j\| < -3\delta + \delta + \delta = - \delta ,
\end{align*}
which shows that the linear functional determined by $\mu$ is not positive.

We have established that ${\mathbf{i}} : M(K') \to X$ is a bipositive surjective isometry. By \cite{Ab}, ${\mathbf{i}}$ is a lattice isometry.
\end{proof}



We shall denote by $\A_1$ the set of normalized atoms of $X^*$. By Lemma \ref{l:am-dual}, $X^* = M(K')$, hence $\A_1 = \{\delta_t : t \in K'\} \subset \ball(X^*)_+$.
Below we show that $\A_1$ (equipped with the weak$^*$ topology inherited from $X^*$) is topologically homeomorphic to $K'$.


\begin{lemma}\label{l:sual-atoms}
The map ${\mathbf{j}} : K' \to \A_1 : t \mapsto \delta_t$ is a homeomorphism.
\end{lemma}

\begin{proof}
To establish the continuity of ${\mathbf{j}}$, suppose the net $t_\alpha$ converges to $t$ in $K'$.
By continuity, $\delta_{t_\alpha}(x) = x(t_\alpha) \to x(t) = \delta_t(x)$ for any $x \in X$, hence $\delta_{t_\alpha} \to \delta_t$ in the weak$^*$ topology.

For the continuity of ${\mathbf{j}}^{-1}$, consider a net $(t_\alpha) \subset \A_1$ so that $\delta_{t_\alpha} \to \delta_t \in \A_1$ in the weak$^*$ topology -- that is, $x(t_\alpha) \to x(t)$ for any $x \in X$. 
By the compactness of $K$, it suffices to show that the limit of any convergent subnet of $(t_\alpha)$ is $t$.

Suppose $(t'_\beta)$ is a subnet of $(t_\alpha)$, which converges to $s \in K$.
Then for any $x \in X$, we have $x(s) =  \lim_\beta x(t'_\beta) = x(t)$.
As $x(t)$ is not always $0$, part (2) of Definition \ref{d:benyamini} implies $s \neq \infty$.
Further, $x(t) = x(s)$ for any $x \in X$, hence parts (3) and (4) of Definition \ref{d:benyamini} show that $t=s$.

%
%
%
%
\end{proof}

\section{Renormings of Benyamini spaces}\label{s:am_renormings_benyamini}

\begin{theorem}\label{t:new_theorem_about_Benyamini}
 Suppose $(X, \| \cdot \|)$ is a Benyamini space. Then, for any $c > 1$, $X$ can be equipped with an equivalent norm $\tri{ \cdot }$ so that $\| \cdot \| \leq \tri{ \cdot } \leq c^2\| \cdot \|$, so that the identity is the only lattice isometry on $(X, \tri{ \cdot })$.
 If $X$ has no more than one atom, then $\tri{ \cdot }$ can be selected to be an AM-norm.
\end{theorem}

\begin{remark}\label{r:importance_of_number_of_atoms}
 The restriction on the number of atoms is essential here; see Remark \ref{r:assumptions_needed}.
\end{remark}

The rest of this section is devoted to proving \Cref{t:new_theorem_about_Benyamini}.

Assume that $X$ is a $C$-Benyamini space ($C<2$) and that  $c < \sqrt[3]{C}$.
Let $A$ and $B$ be the sets of all $n \in \N$ for which $K_n'$ is infinite, resp.~finite and non-empty.
For $n \in B$, write $K_n' = \{t_{1n}, \ldots, t_{p_n n}\}$. For $n \in A$, find a sequence $t_{1n}, t_{2n}, \ldots $ of distinct elements of $K_n'$ which is dense in $K_n'$.
Find a family $(\lambda_{in})_{n \in A \cup B} \subset (1,c)$ of distinct numbers so that:
(i) for $n \in A$, $c > \lambda_{1n} > \lambda_{2n} > \ldots$, and $\lim_i \lambda_{in} = 1$; (ii) for $n \in B$, $c > \lambda_{1n} > \ldots> \lambda_{p_n n} > 1$.
For each $t \in K'$, let $\mu(t) = \lambda_{in}$ if $t = t_{in}$ for some $i$ and $n$, $\mu(t) = 1$ otherwise.

Denote the normalized atoms of $X$ by $(\theta_i)_{i \in I}$, where the set $I$ is countable. By Proposition \ref{p:atomes_in_benyamini}, each $\theta_i$ corresponds with a hereditarily isolated point $a_i \in K'$.
Furthermore, for each $i$, there exists a canonical band projection $P_i$ onto $\spn[\theta_i]$. Then $P_i x = x(a_i) \theta_i$. 

Our definition of $\tri{ \cdot }$ would depend on the cardinality of $I$.

\underline{$|I|=0$.} For $x \in X$ set
\begin{equation}
\triple{x} = \sup_{t \in K'} \mu(t) |x(t)| .
\label{eq:am norm: 0 atoms}
\end{equation}

\underline{$|I|=1$.} Write $I = \{1\}$; represent $X$ as $X_1 \oplus \R$, where $X_1 = \ker P_1$ is a $C$-Benyamini space (with the underlying space obtained by removing from $K$ all the points $\phi_{nm}(a_1)$, when $a_1 \in K_n$ and $m \geq n$). Let $\tri{ \cdot }_1$ be the norm defined on $X_1$ using \eqref{eq:am norm: 0 atoms} (with some collection $(t_{ni})$). Let
\begin{equation}
\triple{x} = \max \big\{ \tri{(I-P_1)x}_1 , \|P_1 x\| \big\} .
\label{eq:am norm: 1 atom}
\end{equation}

\underline{$|I|>1$.} Write $I = \{1, \ldots, m\}$ ($2 \leq m < \infty$) or $I = \N$.
Let ${\mathcal{P}} = \{ (i, j) \in I^2 : i < j \}$, and let $\pi : {\mathcal{P}} \to \N$ be an injection.
For $(i,j) \in {\mathcal{P}}$, let $\| \cdot \|_{i,j}$ be the norm on $\R^2$ whose unit ball is an octagon with vertices
 $$
 \Bigg( \pm \bigg( 1 - \frac{c-1}{c (2\pi(i,j)+1)} \bigg), \pm 1 \Bigg)
 {\textrm{   and   }}
 \Bigg( \pm 1, \pm \bigg( 1 - \frac{c-1}{2c \pi(i,j)} \bigg) \Bigg)
 $$
We mention some properties of the norms $\| \cdot \|_{i,j}$, to be used in the sequel.

\begin{labeling}{N3}
\item [N1]
$\| \cdot \|_\infty \leq \| \cdot \|_{i,j} \leq c \| \cdot \|_\infty$.
\item [N2]
The formal identity $(\R^2, \| \cdot \|_{i_1,j_1}) \to (\R^2, \| \cdot \|_{i_2,j_2})$ (with the first vector of the canonical basis mapping to the first, and the second -- to the second) is an isometry iff $i_1 = i_2$ and $j_1 = j_2$. This follows from a comparison of extreme points.
\item [N3] 
For $\gamma > 1$ and $k \in I$, there exists $L = L(k,\gamma) \geq k$ so that $\| \cdot \|_{k,j} \leq \gamma \| \cdot \|_\infty$ for $j > L$.
\item [N4] 
For $\gamma > 1$,
there exists $M = M(\gamma)$ so that $\| \cdot \|_{i,j} \leq \gamma \| \cdot \|_\infty$ whenever $j > i > M$.
\item [N5] 
If $|\alpha| \vee |\beta| = 1$ and $|\alpha| \wedge |\beta| \leq 1/c$, then $\| (\alpha,\beta) \|_{ij} = 1$.
\end{labeling}
 
 
 We let
 \begin{equation}
\triple{x} = \max \Big\{ \sup_{t \in K'} \mu(t) |x(t)| ,
 \sup_{(i,j) \in {\mathcal{P}}} \big\| \big( \mu(a_i) x(a_i), \mu(a_j) x(a_j) \big) \big\|_{i,j} \Big\} .
\label{eq:am norm: >1 atoms}
\end{equation}

Clearly, we always have $\| \cdot \| \leq \tri{ \cdot } \leq c^2 \| \cdot \|$ (in fact, if $|I| \leq 1$, we can replace $c^2$ by $c$). 
It is also clear that for $|I| \leq 1$, $\tri{ \cdot }$ is an AM-norm.
To show that the only lattice isometry on $(X, \tri{ \cdot })$ is the trivial one, we need a series of lemmas. As the proof for $|I|=1$ follows immediately from that for $|I| = 0$, we shall only consider the cases of $I = \emptyset$ and $|I| \geq 2$.

First we establish the norms of point masses. 
Let $\hat{\delta}_t = \mu(t) \delta_t$.

\begin{lemma}\label{l:norms of point masses: all cases}
For any $t \in K'$, $\triple{\hat{\delta}_t} = 1$.
\end{lemma}

\begin{proof}
For $x \in X$ and $t \in K'$, we clearly have $\tri{x} \geq \mu(t) |x(t)| = \big| \hat{\delta}_t(x) \big|$, hence $\triple{\hat{\delta}_t} \leq 1$. It remains to prove the opposite inequality.

Fix $t \in K'$ and $\gamma > 1$. We need to find $x \in X_+$ such that $x(t) = 1/\mu(t)$ and $\triple{x} \leq \gamma$.
To this end, find $n$ so that $t \in K_n'$. Next, construct a finite set $V \subset K_n'$ consisting of ``potentially troublemaking'' points.
If $|I| = \emptyset$, let
$$
V = \{s \in K_n' : \mu(s) > \gamma \mu(t)\} .
$$
If $|I| \geq 2$ and $t$ is not hereditarily isolated, let
$$
V = \{s \in K_n' : \mu(s) > \gamma \mu(t)\} \cup \{a_i \in K_n' : i \leq M(\gamma) \} ,
$$
with $M(\gamma)$ as in [N4].

If $|I| \geq 2$ and $t$ is hereditarily isolated, then $t = a_k$ for some $k$. Let
$$
V = \{s \in K_n' : \mu(s) > \gamma \mu(t)\} \cup \{a_i \in K_n' : i \leq M(\gamma) \vee  L(k,\gamma) \} \backslash \{a_k\} ,
$$
where $L(k, \gamma )$ comes from property [N3].

The set $V$ is finite and does not contain $t$.
Find an open set $U \subset K_n' \backslash V$ containing $t$. Find $x \in C(K_n)$ such that $x$ vanishes outside of $U$ and $0 \leq x \leq 1/\mu(t) = x(t)$. 
Define $x$ to be $0$ on $K_m$ for $m < n$. This function is consistent, so by Proposition \ref{p:consistent_extension}, there exists $\tilde{x} \in X_+$ so that $\tilde{x}|_{K_1 \cup \ldots \cup K_n} = x$ and $\|\tilde{x}\| = 1/\lambda_{in}$.

It remains to show that $\triple{\tilde{x}} \leq \gamma^2$. This will follow if we establish that
\begin{equation}
 \mu(s) |\tilde{x}(s)| \leq \gamma   {\textrm{   for  any   }}  s \in K' ,
 \label{eq:pointwise inequality}
\end{equation}
and (in the case of $|I| \geq 2$)
\begin{equation}
 \big\| \big( \mu(a_i) \tilde{x}(a_i), \mu(a_j) \tilde{x}(a_j) \big) \big\|_{i,j} \leq \gamma^2 
  {\textrm{   for  any   }}  i < j .
 \label{eq:pairwise inequality}
\end{equation}

Note that, due to our construction of $\tilde{x}$, $\tilde{x}(s) = 0$ if $s \in K'_m$ with $m < n$.
For $s \in K_n'$, we have $\tilde{x}(s) = 0$ for $s \notin U$, while for $s \in U$, $\mu(s) \leq \gamma \mu(t)$, so $\mu(s) |\tilde{x}(s)| \leq \gamma$. Finally, if $s \in K_m'$ for some $m > n$, we have
$\tilde{x}(s) \leq C^{n-m}/\mu(t)$, hence $\mu(s) |\tilde{x}(s)| \leq c/C < 1 < \gamma$.
This establishes \eqref{eq:pointwise inequality}.

To handle \eqref{eq:pairwise inequality}, note that
if $a_i \in \cup_{m<n} K'_m \cup (K_n' \backslash U)$, then $\tilde{x}(a_i) = 0$, and therefore,
$$
\big\| \big( \mu(a_i) \tilde{x}(a_i), \mu(a_j) \tilde{x}(a_j) \big) \big\|_{i,j}
= \big\| \big( 0, \mu(a_j) \tilde{x}(a_j) \big) \big\|_{i,j}
= \mu(a_j) \tilde{x}(a_j) .
$$
The right hand side cannot exceed $\gamma$, as discussed in the paragraph relating to \eqref{eq:pointwise inequality}.
The same conclusion holds if $a_j \in \cup_{m<n} K'_m \cup (K_n' \backslash U)$.

If $a_i, a_j \in \cup_{\ell > n}K'_\ell$, then $\tilde{x}(a_i), \tilde{x}(a_j) \leq 1/(\mu(t) C)$, hence
$$
\big\| \big( \mu(a_i) \tilde{x}(a_i), \mu(a_j) \tilde{x}(a_j) \big) \big\|_{i,j} \leq \frac{c^2}{\mu(t) C} < 1 .
$$

Now consider the case of $a_i \in U$, $a_j \in \cup_{\ell > n} K_\ell'$. In this situation, $\mu(a_j) \tilde{x}(a_j) < c/C < c^{-2}$, hence, by [N5], 
$$
\big\| \big( \mu(a_i) \tilde{x}(a_i), \mu(a_j) \tilde{x}(a_j) \big) \big\|_{i,j} \leq \gamma .
$$
The same conclusion holds if $a_j \in U$, $a_i \in \cup_{\ell > n} K_\ell'$.

Finally, if $a_i, a_j \in U$, then $\mu(a_i), \mu(a_j) \leq \gamma \mu(t)$. By the choice of $U$,
$$
\big\| \big( \mu(a_i) \tilde{x}(a_i), \mu(a_j) \tilde{x}(a_j) \big) \big\|_{i,j} \leq 
\gamma \big\| \big( \mu(a_i) \tilde{x}(a_i), \mu(a_j) \tilde{x}(a_j) \big) \big\|_\infty \leq \gamma^2 .
$$
The same conclusion holds if the roles of $a_i$ and $a_j$ are reversed. We have now established \eqref{eq:pairwise inequality}.
\end{proof}

Now suppose $T$ is a surjective lattice isometry on $(X, \triple{ \cdot })$. Note first that $T$ fixes the atoms of $X$:

\begin{lemma}\label{l:T fixes atoms}
 For any $i \in I$, $T \theta_i = \theta_i$.
\end{lemma}

\begin{proof}
This is obvious if $|I| \leq 1$. For $|I| \geq 2$, let $e_i = \theta_i/\mu(a_i)$ be the normalized atoms. By \eqref{eq:am norm: >1 atoms}, for any $\alpha, \beta \in \R$, we have
$$
\tri{ \alpha e_i + \beta e_j } = \big\| (\alpha, \beta) \big\|_{i,j} .
$$
If $T$ maps $e_i$ and $e_j$ to $e_k$ and $e_\ell$ respectively, then
$$
\big\| (\alpha, \beta) \big\|_{i,j} = \big\| (\alpha, \beta) \big\|_{k, \ell} 
{\textrm{  for  any  }}  \alpha, \beta ,
$$
which, in light of Property [N2], implies $i=k$, $j=\ell$.
\end{proof}

Now observe that $T^*$ is interval preserving \cite[Theorem 1.4.19]{M-N}, hence it maps atoms of $X^*$ to atoms.
The atoms in $X^*$ are characterized by Proposition \ref{p:atomes_in_benyamini}.
By Lemma \ref{l:norms of point masses: all cases}, the set of normalized atoms of $(X^*, \tri{ \cdot })$ (which we shall denote by $\A$) coincides with $\big\{ \hat{\delta}_t : t \in K' \big\}$.

Thus, by Lemma \ref{l:norms of point masses: all cases}, there exists a bijection $\psi : K' \to K'$ so that $T^* \hat{\delta}_t = \hat{\delta}_{\psi(t)}$.
We shall show that $\psi(t) = t$ is the identity map. In fact, Lemma \ref{l:T fixes atoms} already shows that $\psi(t) = t$ if $t$ is a hereditarily isolated point.

To proceed further, in the next few lemmas we examine weak$^*$ convergence in $\A$. For convenience, we denote by $\phi_{nn}$ the identity map on $D(n,n):=K_n$.

\begin{lemma}\label{l:convergence_in_Km new}
Suppose $m, n \in \N$, $t \in K_n'$, and the sequence $(t_i) \subset K_m' \backslash \{t\}$ converges to $s$.
Then the following are equivalent:
\begin{enumerate}
 \item $m \geq n$, and $s = \phi_{nm}(t)$.
 \item
 ${\mathrm{w}}^*-\lim_i \hat{\delta}_{t_i} = \alpha \hat{\delta}_t$ for some $\alpha > 0$.
\end{enumerate}
Moreover, if (1) holds, then 
(2) holds with $\alpha = C^{n-m}/\mu(t)$.
\end{lemma}

\begin{proof}
To show that (1) implies (2), as well as the ``moreover'' statement, we only need to observe that, due to our selection of $(\lambda_{jm})$, we have $\lim_i \mu(t_i) = 1$.   
We need to establish the converse.

First show that $m \geq n$. If $m < n$, then find an open set $U \subset K_n'$ containing $t$. By Proposition \ref{p:consistent_extension}, there exists $x \in X$ so that $0 \leq x \leq 1 = x(t)$, which vanishes on $K_n \backslash U$ and on $K_j$ for $j < n$.
In particular, $\hat{\delta}_t(x) \neq 0$, while $\hat{\delta}_{t_i}(x) = 0$ for any $i$. This  contradicts (2).

Thus $m \geq n$. Next show that $t \in D(n,m)$ 
and $s = \phi_{nm}(t)$. 
Suppose, for the sake of contradiction, that either $t \notin D(n,m)$, or $t \in D(n,m)$ and $s \neq \phi_{nm}(t)$. Find the smallest $i\leq m$ so that $s \in D(m,i)$, and let $s' = \phi_{mi}(s)$.
Then $t \neq s'$. 
By Lemma \ref{l:extend_from_two roots}, there exists $x \in X$ so that $x(t) = 1$ and $x(s') = 0$, hence also $x(s) = 0$. We observe that $\hat{\delta}_t(x) \neq 0$ and $\lim_i \hat{\delta}_{t_i}(x) = 0$, again contradicting (2).
%
%
\end{proof}

\begin{lemma}\label{l:convergence in general new}
Suppose we are given $t \in K_n'$ and a sequence $(t_i) \subset K' \backslash \{t\}$.
Then the following are equivalent:
\begin{enumerate}
 \item There exists $m \geq n$ so that for $i$ large enough, $t_i \in K_m'$. Furthermore, $(t_i)$ converges to $s = \phi_{nm}(t)$.
 \item
 ${\mathrm{w}}^*-\lim_i \hat{\delta}_{t_i} = \alpha \hat{\delta}_t$ for some $\alpha > 0$.
\end{enumerate}
Moreover, if (1) holds, then, in (2),
$\alpha = C^{n-m}/\mu(t)$.
\end{lemma}

\begin{proof}
Lemma  \ref{l:convergence_in_Km new} shows that (1) implies (2), as well as the ``moreover'' conclusion.
To establish (2) $\Rightarrow$ (1), find, for each $i$, $m(i) \in \N$ so that $t_i \in K_{m(i)}'$.
We shall show that the sequence $(m(i))$ is eventually constant.

First we show that $(m(i))$ is bounded. Indeed, otherwise we can find a sequence $(i_p)$ so that $\lim_p m(i_p) = \infty$. Clearly $\lim x(t_{i_p}) = 0$ for any $x \in X$, hence $\hat{\delta}_{i_p} \overset{w^*}{\to} 0$.

Now suppose, for the sake of contradiction, that $(m(i))$ does not stabilize. Passing to a subsequence, we can assume that there exists $m_1 \neq m_2$ so that $m(i) = m_1$ if $i$ is odd, and $m(i) = m_2$ is even if $i$ is even. Further, we can assume that $(t_{2i-1})$ and $(t_{2i})$ converge to $s_1 \in K_{m_1}$ and $s_2 \in K_{m_2}$, respectively.
From Lemma \ref{l:convergence_in_Km new}, $m_1, m_2 \geq n$, $t_{2i} \to s_2 = \phi_{m_2 n}(t)$, and
${\mathrm{w}}^*-\lim_i \hat{\delta}_{t_i} = \hat{\delta}_t/(C^{m_2-n} \mu(t))$.
Similarly, $t_{2i-1} \to s_1 = \phi_{m_1 n}(t)$, and
${\mathrm{w}}^*-\lim_i \hat{\delta}_{t_i} = \hat{\delta}_t/(C^{m_1-n} \mu(t))$.
Thus, $1/\alpha = C^{m_2-n} \mu(t) = C^{m_1-n} \mu(t)$, which leads to the impossible conclusion $m_1 = m_2$.

Thus, the sequence $(m(i))$ is eventually constant. To conclude the proof, invoke Lemma \ref{l:convergence_in_Km new}.
\end{proof}

\begin{lemma}\label{l:convergent sequence new}
 Suppose $t \in K'$ is not hereditarily isolated. Then there exists a sequence $(t_i) \subset K'$ so that $\hat{\delta}_{t_i} \overset{w^*}{\to} \alpha \hat{\delta}_t$, for some $\alpha \in (0,1]$.
 Moreover, for every such sequence there exists $r \in \{0,1,2,\ldots\}$ so that $\alpha = 1/(C^r\mu(t))$.
\end{lemma}

\begin{proof}
Suppose first $t$ is not isolated in $K_n$. Then $t$ cannot be isolated in the open subset $K_n' \subset K$, so we can find a sequence $(t_i) \subset K_n'$, converging to $t$. Clearly $\delta_{t_i} \to \delta_t$ (in the weak$^*$ topology).
Moreover, $\mu(t_i) \to 1$, hence $\hat{\delta}_{t_i} \to \alpha \hat{\delta}_t$, where $\alpha = 1/\mu(t) \in (1/c,1]$.

Now suppose $t$ is isolated in $K_n$ (equivalently, in $K_n'$). Use Proposition \ref{p:atomes_in_benyamini} to find the smallest $m > n$ s.t. $s = \phi_{nm}(t)$ is not isolated in $K_m$. 
We claim that $K_m'$ is non-empty, and $s$ belongs to the closure.
Indeed, 
as $t \in K_n'$, $s$ cannot belong to $D(m,k)$ with $k < n$.
In addition, if $s \in D(m,k)$ for some $n \leq k < m$, then $s$ is an isolated point of $D(m,k)$, due to the minimality of $m$.
Consequently, $s$ is an isolated point of $\cup_{k < m} D(m,k)$.
As $s$ is not isolated in $K_m$, we can find a sequence $(t_i) \subset K_m'$ converging to $t$. Then $\delta_{t_i} \overset{w^*}{\to} C^{n-m} \delta_t$, hence $\hat{\delta}_{t_i} \to \alpha \hat{\delta}_t$, where $\alpha = C^{n-m}/\mu(t) \in (C^{n-m}/c,C^{n-m}]$.

Now suppose $\hat{\delta}_{t_i} \overset{w^*}{\to} \alpha \hat{\delta}_t$, for some $\alpha \in (0,1]$. By Lemma \ref{l:convergence in general new}, there exists $m$ so that $t_i \in K_m$, for $m$ large enough; and furthermore, $t_i \to \phi_{mn}(t)$.
As in the previous paragraph, $\alpha = C^{n-m}/\mu(t)$.
\end{proof}

\begin{proof}[Theorem \ref{t:new_theorem_about_Benyamini} -- completion of the proof]
 Suppose $T$ is a lattice isometry on $(X, \triple{ \cdot })$. 
 By Subsection \ref{ss:dual}, it suffices to show that $T^* \hat{\delta}_t = \hat{\delta}_t$ for any $t \in K'$. As $T^*$ maps normalized atoms to normalized atoms, $T^* \hat{\delta}_t = \hat{\delta}_s$, where $s = \psi(t) \in K'$. By Lemma \ref{l:T fixes atoms}, $\psi(t) = t$ if $t$ is hereditarily isolated. As the set $\A$ of normalized atoms is identified with $\big\{ \hat{\delta}_t : t \in K' \big\}$, we conclude that $t$ is not hereditarily isolated iff $\psi(t)$ satisfies the same condition.
 For future use, note that if $t$ is hereditarily isolated, then $t = t_{in}$ for some $i,n$.
 
 Now suppose $t$ is not hereditarily isolated. Let $s = \psi(t)$.
 In light of Lemma \ref{l:convergent sequence new}, there exists a sequence $(u_i) \subset K'$ so that $\hat{\delta}_{u_i} \overset{w^*}{\to} \alpha \hat{\delta}_t$. Moreover, for every such sequence,
$$
\frac{1}{\mu(t)} = \nu(t) := \sup \big\{ C^k \alpha : k \in \{0,1,2,\ldots\}, C^k \alpha \leq 1 \big\} .
$$
Being isometric and weak$^*$ to weak$^*$ continuous, $T^*$ preserves $\nu( \cdot)$, hence $\mu( \psi(t) ) = \mu(t)$, for any $t \in K'$.

Recall that $t_{in}$ is the unique point $t$ with $\mu(t) = \lambda_{in}$.
Consequently, $\psi(t_{in}) = t_{in}$, or equivalently, $T^* \hat{\delta}_{t_{in}} = \hat{\delta}_{t_{in}}$.

Now suppose $t \in K' \backslash ( \cup_{i,n} \{ t_{in} \} )$ is not hereditarily isolated.
Find a sequence $\big(t_{i_jn_j}\big)_j$ which converges to $\phi_{mn}(t)$ for some $m \geq n$.
By Lemma \ref{l:convergence_in_Km new}, 
$$
{\mathrm{w}}^*-\lim_j \hat{\delta}_{t_{i_jn_j}} = C^{n-m} \hat{\delta}_t ,
$$
hence, due to the weak$^*$ to weak$^*$ continuity of $T^*$,
$$
{\mathrm{w}}^*-\lim_j T^* \hat{\delta}_{t_{i_jn_j}} = C^{n-m} \hat{\delta}_{\psi(t)} ,
$$
However, the left hand sides of the two centered expressions coincide, hence $\psi(t) = t$.
\end{proof}

\begin{remark}\label{r:assumptions_needed}
 In Theorem \ref{t:new_theorem_about_Benyamini}, the desired renorming cannot be an AM-space if the number of atoms exceeds $1$. Indeed, suppose $a_1, \ldots, a_n$ are normalized atoms in an AM-space $X$, and let $X_0 = \{a_1, \ldots, a_n\}^\perp$. If $\pi$ is a permutation of $\{1, \ldots, n\}$, then $T : X \to X$, defined by $T a_i = a_{\pi(i)}$ and $Tx = x$ for $x \in X_0$, is an isometry. Thus, any AM renorming of a space with more than one atom will have non-trivial lattice isometries.
 \end{remark}


\begin{thebibliography}{22}
	
	 \bibitem{Ab} Y.~Abramovich.
	Isometries of normed lattices (in Russian).
	Optimizatsiya 43(60) (1988), 74--80. 
	

\bibitem{Be} S. Bellenot. Banach spaces with trivial isometries.
Israel J. Math. 56 (1986), 89--96. 

	\bibitem{Ben} Y. Benyamini.
	Separable $G$ spaces are isomorphic to $C(K)$ spaces.
	Israel J. Math, 14 (1973), 287-293.
	
	\bibitem{Diestel75} J.~Diestel.
	Geometry of Banach spaces -- selected topics.
 Springer, Berlin, 1975.
 
 \bibitem{FA_inf_dim_geometry}
 M.~Fabian, P.~Habala, P.~Hajek, V.~Montesinos Santalucia, J.~Pelant, and V.~Zizler.
Functional analysis and infinite-dimensional geometry. 
Springer, New York, 2001.
	
 \bibitem{BS theory}
 M.~Fabian, P.~Habala, P.~Hajek, V.~Montesinos Santalucia, and V.~Zizler.
Banach space theory. The basis for linear and nonlinear analysis. 
Springer, New York, 2011.
	
\bibitem{FeG_TAMS10} V. Ferenczi and E.M. Galego.
Countable groups of isometries on Banach spaces.
Trans. Amer. Math. Soc. 362 (2010), 4385--4431. 

\bibitem{FeR_EM11} V. Ferenczi and C. Rosendal.
Displaying Polish groups on separable Banach spaces.
Extracta Math. 26 (2011), 195--233. 

\bibitem{FeR_DU13} V. Ferenczi and C. Rosendal.
On isometry groups and maximal symmetry.
Duke Math. J. 162 (2013), 1771--1831.

\bibitem{Ja} K. Jarosz.
Any Banach space has an equivalent norm with trivial isometries.
Israel J. Math. 64 (1988), 49--56. 

\bibitem{LT2} J. Lindenstrauss and L. Tzafriri.
Classical Banach spaces II. Springer, Berlin, 1979.

\bibitem{M-N} P. Meyer-Nieberg.
Banach lattices. Springer, Berlin, 1991.
	
\end{thebibliography}
\end{document}